\documentclass[12pt,a4paper]{article}
\usepackage[top=3cm, bottom=2.5cm, right=2.5cm, left=2.5cm]{geometry}

\usepackage[]{hyperref} 

\usepackage[utf8]{inputenc}
\usepackage{amsmath}
\usepackage{amsthm}
\usepackage{amssymb}
\usepackage{mathrsfs}
\usepackage{amsfonts}
\usepackage{bbold}

\usepackage{algorithm}
\usepackage{algorithmicx}
\usepackage{algpseudocode}
\usepackage{graphicx}
\usepackage{xifthen}
\usepackage{rotating}
\usepackage{xcolor}
\usepackage{subfig} 
\usepackage{psfrag}

\usepackage{amsopn}

\newcommand{\hx}{\hat{\x}}
\newcommand{\hT}{\hat{T}}
\newcommand{\hphi}{\hat{\phi}}

\newcommand{\hi}{\hat{i}}
\newcommand{\MBu}{\MB{u}}
\newcommand{\MBv}{\MB{v}}

\newcommand{\hMBB}{\hat{\MB{B}}}
\newcommand{\MBA}{\MB{A}}
\newcommand{\Mu}{\M{u}}
\newcommand{\Mv}{\M{v}}
\newcommand{\hMB}{\hat{\M{B}}}
\newcommand{\MA}{\M{A}}

\newcommand{\skipifemptyarg}[1]{\ifthenelse{\isempty{#1}}{}{\left[#1\right]}}
\newcommand{\skipifscalar}[1]{\ifthenelse{\isempty{#1}}{}{;#1}}

\newcommand{\bs}[1]{\boldsymbol{#1}}
\newcommand{\scal}[2]{\bigl(#1,#2\bigr)}
\newcommand{\set}[1]{\mathbb{#1}} 
\newcommand{\M}[1]{\mathsf{#1}} 
\newcommand{\MB}[1]{\bs{\M{#1}}} 
\newcommand{\V}[1]{{\bs{#1}}} 
\newcommand{\T}[1]{\bs{#1}} 

\newcommand{\x}{\V{x}}
\newcommand{\y}{\V{y}}
\newcommand{\sV}{\mathcal{V}}
\newcommand{\sW}{\mathcal{W}}
\newcommand{\Di}{\mathcal{D}}

\newcommand{\alp}{\ensuremath{\alpha}}

\newcommand{\del}{\ensuremath{\delta}}

\newcommand{\sR}{\set{R}}
\newcommand{\sN}{\set{N}}

\newcommand{\D}[1]{\,{\mathrm d}#1}

\newcommand{\sep}{\,|\,}
\newcommand{\bilf}[2]{a(#1, #2)}

\newcommand{\mesh}{\mathcal{M}} 
\newcommand{\meshN}{\mathcal{M}_\VN} 
\newcommand{\sP}{\mathcal{P}} 
\newcommand{\fC}{\mathcal{C}} 
\newcommand{\fD}{\mathcal{D}} 
\newcommand{\VN}{{\V{N}}}
\DeclareMathOperator{\nnz}{nnz}
\DeclareMathOperator{\nrows}{nrows}
\newcommand{\diag}{^\mathrm{DoGIP}}

\newcommand{\Vv}{{\V{v}}}
\newcommand{\Vu}{{\V{u}}}
\DeclareMathOperator{\mem}{mem}

\theoremstyle{plain}
\newtheorem{theorem}{Theorem}
\newtheorem{definition}[theorem]{Definition}
\newtheorem{lemma}[theorem]{Lemma}
\newtheorem{corollary}[theorem]{Corollary}
\newtheorem{remark}[theorem]{Remark}

\title{Double-grid quadrature with interpolation-projection (DoGIP) as a novel discretisation approach: An application to FEM on simplexes}
\usepackage{authblk}
\author[1]{Jaroslav~Vondřejc}
\affil[1]{Technische Universit\"{a}t Braunschweig, Institute of Scientific Computing, Mühlenpfordtstrasse~23, 38106 Braunschweig, Germany, \href{mailto:j.vondrejc@tu-bs.de}{j.vondrejc@tu-bs.de}, \href{mailto:vondrejc@gmail.com}{vondrejc@gmail.com}}
\date{\today}

\begin{document}
\maketitle

\begin{abstract}
This paper is focused on the double-grid integration with interpolation-projection (DoGIP), which is a novel matrix-free discretisation method of variational formulations introduced for Fourier--Galerkin approximation.
Here, it is described as a more general approach with an application to the finite element method (FEM) on simplexes.
The approach is based on treating the trial and a test function in variational formulation together, which leads to the decomposition of a linear system into interpolation and (block) diagonal matrices.
It usually leads to reduced memory demands, especially for higher-order basis functions, but with higher computational requirements.
The numerical examples are studied here for two variational formulations: weighted projection and scalar elliptic problem modelling, e.g. diffusion or stationary heat transfer.
This paper also opens a room for further investigation, which is discussed in the conclusion.

\end{abstract}

\section{Introduction}
\label{sec:introduction}
A discretisation of non-linear partial differential equations, after linearisation, leads to the solution of linear systems.
Often, they require excessive memory, especially when the system matrices are fully populated and no special structure is incorporated.
This paper focuses on an alternative matrix-decomposition approach that is based on a double-grid integration with interpolation-projection operator (DoGIP).
This method, already used for numerical homogenisation with the Fourier--Galerkin method \cite{VoZeMa2014GarBounds,Vondrejc2015FFTimproved,Vondrejc2016DoGIP}, belongs to the matrix-free approaches with possible applicability to the discretisation of variational problems with various approximation spaces.
Here, the focus is aimed at the finite element method (FEM) on simplexes (e.g.\ triangles or tetrahedra).
\paragraph{Introduction of the DoGIP idea}
As a model problem, an abstract variational setting discretised with Galerkin approximation is assumed: find trial function $u$ from some finite dimensional vector space $\sV$ such that
\begin{align*}
a(u,v) = F(v)
\quad\forall v\in \sV.
\end{align*}
This problem, which naturally arises in many engineering applications, is defined via a positive definite continuous bilinear form $a:\sV\times \sV\rightarrow\sR$ and linear functional $F:\sV\rightarrow\sR$ representing a source term.
The solution, 
which is expressed as a linear combination of basis functions $u(\x)=\sum_{i=1}^V \Mu_i\phi_i(\x)$ with $V=\dim\sV$, is determined by the vector of coefficients $\MBu = (\Mu_i)_{i=1}^V$ that can be calculated from a corresponding linear system
\begin{align*}
\MB{A} \MBu &= \MB{b},
&
\text{where }
\MB{A}= \bigl( a(\phi_j, \phi_i) \bigr)_{i,j=1}^V \in\sR^{V\times V}
\text{ and }
\MB{b}=\bigl(F(\phi_i)\bigr)_{i=1}^V\in\sR^{V}.
\end{align*}

The properties of the matrix $\MB{A}$ are fundamental for the decision about a linear solver and consequently for computational and memory requirements.
The FEM builds on the basis functions with a local support, which results in a sparse matrix.

The idea in the DoGIP approach, first used for the spectral Fourier--Galerkin method \cite{Vondrejc2015FFTimproved,VoZeMa2014GarBounds}, is based on expressing the trial and the test function as a product on a double-grid space, i.e.\ $uv\in\sW$ or $\nabla u\odot\nabla v$ in the case of the weighted projection and the scalar elliptic equation with differential operator.
Particularly, the double-grid space $\sW$ is composed of the polynomials with the doubled order.

The DoGIP approach allows to decompose the linear system
\begin{align*}
\MB{A} &= \MB{C}^*\MB{A}\diag \MB{B},
&
\text{where }
\MB{A}\diag\in\sR^{W\times W},\MB{B},\MB{C}\in\sR^{W\times V} \text{ with }W>V.
\end{align*}
Here, the diagonal (or block-diagonal) matrix $\MB{A}\diag$ expresses the weights obtained by integration of material-like coefficients with respect to a basis of the double-grid space.
Despite the bigger double-grid size $W>V$, the memory requirements to store $\MB{A}\diag$ are typically reduced compared to the original matrix $\MB{A}$. Usually it is impossible to fully avoid assembling matrix $\MBA$ because of the high number of integration points for general material coefficients, whose evaluation can be also computationally demanding.

Matrices $\MB{B},\MB{C}$, that can equal one another, represent interpolation-projection operators from the original to the double-grid space and depend on the approximation spaces only.
Those matrices are not stored because the required matrix-vector multiplication can be substituted with an efficient numerical routine.
Particularly, the interpolation-projection matrix in the Fourier--Galerkin method is evaluated very efficiently with the fast Fourier transform using $\mathcal{O}(N\log{}N)$ operations \cite{Vondrejc2015FFTimproved,Vondrejc2016DoGIP}.
In case of FEM, only the interpolation on the reference element is needed.

\paragraph{Alternative approaches.}
There are also many other matrix-free evaluations of discretised differential operators \cite{Orszag1980,Kronbichler2012,Cantwell2011,Deville2002book,Huismann2017}, which often incorporate orthogonal basis functions within the spectral methods.
Instead of the assembling of the system matrix $\MBA$, these approaches typically rely on the evaluation of integrals on the fly by fast sum factorisation techniques over numerical integration points.
It can be demonstrated for the generalised Laplacian with positive definite material coefficients $\T{A}$, i.e.
 $\bilf{u}{v} = \sum_T\int_T\T{A}(\x)\nabla u(\x)\cdot \nabla v(\x)\D{\x}$ expressed over elements $T$ of a FEM mesh.
This formulation invokes the element-wise evaluation of the system matrix-vector product
$\MBA\MBu = \sum_T\MB{P}_T^*\MBA_T\MB{P}_T\MBu$, where the operator $\MB{P}_T$ extracts the local degrees of freedom $\MBu_T$ from the global vector $\MBu$.
Then the local matrix-vector multiplication $\MBA_T\MBu_T$ is obtained sequentially by evaluating the integral over integration points $\x_q$ with corresponding weights $w_q$, i.e.
\begin{align*}
\nabla u(\x_q) &= J_T^{-\mathrm{T}}(\hx_q)\sum_i\widehat{\nabla}\widehat{\phi}_i(\hx_q)\Mu_{T,i},
\\
(\MBA_T\MBu_T)_l&= \sum_q \T{A}(\x_q) \nabla u(\x_q)\cdot \nabla \phi_l(\x_q) w_q |\det J_T(\x_q)|,
\end{align*}
where the quantities with the hat $\widehat{\cdot}$ are expressed on the reference element.
The matrix-vector product is thus composed of the interpolation of the approximate solution over quadrature points, evaluating Jacobians, material coefficients, integration weights, and finally the adjoint interpolation.

\emph{Discretisation.}
The effectiveness of these approaches depends mainly on the choice of basis functions and of integration points.
A traditional approach builds on the Lagrange basis functions, which can be efficiently combined with a special integration technique.
Particularly the location of the Lagrange basis at the Gauss-Lobatto quadrature points can significantly reduce the computational cost required for the interpolation \cite{Kronbichler2012,Huismann2017,Hesthaven1998}.
The spectral approaches \cite{Gottlieb1977,Boyd2000book,Canuto2006SM_book} build on orthogonal systems of basis functions, which are generally defined on regular domains, e.g. squares or cubes.
The application of the spectral methods to complex geometries \cite{Orszag1980} can utilise a mapping to a regular domain or a multipatch idea leading to the so-called spectral element method \cite{Huismann2017,Pasquetti2006,Karniadakis2013,Karniadakis1990} or fictitious domain methods, see \cite{Cantwell2011,Vos2008,Vos2010} for a comparison.

A general idea that can significantly reduce the computational costs incorporates the tensor product structure of the basis functions $\phi(\x)=\prod_{i=1}^d \phi(x_i)$, which allows to apply interpolation or differentiation operators separately in each direction.
Less attention has been given to the fast evaluation strategies and matrix-free approaches on simplexes (e.g.\ triangles and tetrahedra) because the crucial tensor product structure is more difficult to apply as it requires an additional transform from the simplex to the $d$-dimensional cube \cite{Dubiner1991,Sherwin1995}.
The matrix-free approaches have to be also accompanied with a suitable iterative solver \cite{Luo1998,May2015a,Crivellini2011}. Particularly, an effective parallel implementation is usually required \cite{Alauzet2016}.

\emph{Numerical quadrature.}
Here, integration schemes, over which the matrix-free approaches interpolate, are discussed.
Typically special quadrature can be utilised when an additional information about the integrand is provided.
Particularly, the Gauss-Legendre integration with $k$ points can be used to exactly integrate a polynomial integrand of order up to $2k-1$ in 1D; the corresponding tensor product integration can be used in higher dimensions.
As a general rule, a quadrature is chosen in order to keep the approximation order of the discretised variational formulations.
Namely the consistency error in the first Strang's lemma has to be controlled, which typically assumes a smooth integrand, see \cite{Strang1972varcrime} and \cite[Section~26]{Ciarlet1991HNAFEM}.

The development of an integration rule on the simplexes is more complicated and has been studied in many publications, e.g.\ in \cite{Silvester1970,Grundmann1978,Dunavant1985,Wandzurat2003,Taylor2007,Zhang2009,Witherden2015}.
The Gaussian tensor product integration rule that is degenerated to the simplex leads to the asymmetric distribution of the points with clustering of the points close to one of the vertices.
Although those rules can be applicable to polynomial integrands, a general quadrature requires that (i) points are symmetrically distributed inside of the domain, (ii) integration weights are all positive, and (iii) the truncation error is minimised.
A general approach to compute quadrature points and weights with the required properties can be found in the recent publication \cite{Witherden2015}, which also provides an improvement of the previous integration rules \cite{Dunavant1985,Zhang2009}.
However, an optimal quadrature with respect to the position and the number of the integration points remains an open question.

Although the above discussed quadratures are exact for the polynomials up to some order, they also perform well with general smooth integrands \cite{Novak1996,Trefethen2008,Xiang2012}.
However, their performance is poor for functions with singularities or discontinuities, which occur naturally when the material coefficients under the integral are heterogeneous or when a finite element intersects the boundary of the computational domain.
Those problems appear in the spectral element \cite{Orszag1980,Vos2010} or in the finite cell method \cite{Duster2008,Yang2012,Abedian2013b}
as they avoid the error-prone mesh generation of complicated domains.
Therefore the adaptive integration of rough integrands is needed and can be provided by e.g.\ the local refinement strategy, which allow to keep the exponential convergence \cite{Duster2008,Yang2012,Abedian2013b}.
Nevertheless, the higher number of integration points is required.

\paragraph{Applicability of the DoGIP.} This paper shows that DoGIP is a promising matrix-free approach with various applications.
The method benefits from the fact that only the matrix $\MB{A}\diag$ is stored, which can significantly reduce memory requirements especially for higher order basis functions, and that the interpolation-projection matrix $\MB{B}$ is independent of material coefficients or mesh distortion, which allows an optimisation for fast matrix-vector multiplication.

The method is suitable for discretisation methods with higher order polynomial approximations ($p$-version of FEM \cite{Cantwell2011}, spectral methods \cite{Boyd2000book,Canuto2006SM_book}, finite cell method \cite{Dauge2015}, discontinuous Galerkin methods \cite{Kronbichler2017}, etc.).
Special benefits can be obtained for variational problems when the system matrix has to be solved many times such as in numerical homogenisation \cite{VoZeMa2014FFTH,Schneider2016hexa}, multiscale problems \cite{Ibrahimbegovic2014,Matous2016,Kanoute2009,Abdulle2012a,Gokuzum2017}, optimisation \cite{Sigmund2013approaches}, uncertainty quantification \cite{Ernst2010,Matthies2005}, inverse problems \cite{Stuart2010,VoMa2018CdE}, time-dependent problems \cite{Abedian2013a}, or fluid dynamics \cite{Deville2002book,Fehn2018fluid}.

\paragraph{Outline of the paper}
In section~\ref{sec:dogip-within-finite-element-method}, the methodology of the DoGIP is described for two problems: a weighted projection and a scalar elliptic problem modelling diffusion, stationary heat transfer, etc. Then in section~\ref{sec:numerical-results}, the numerical results are presented along with the discussion of the performance of the DoGIP with respect to the existing matrix-free approaches.

\paragraph{Notation}

In the following text, $L^2(\Omega)$ denotes the space of square Lebesgue-integrable functions $f:\Omega\rightarrow\sR$, where the computational domain $\Omega$ is a bounded open set in $\sR^d$ ($d=2$ or $3$).
The vector-valued and matrix-valued variants of $L^2(\Omega)$ are denoted as $L^2(\Omega;\sR^d)$ and $L^2(\Omega;\sR^{d\times d})$. 
A Sobolev space $H^1(\Omega)$ is a subspace of $L^2(\Omega)$ with gradients in $L^2(\Omega;\sR^d)$. Vector and matrix-valued functions are denoted in bold with lower-case $\V{u}$ and upper-case letters $\T{A}$.
The vectors and matrices storing the coefficients of the discretised vectors are depicted with boldface upright characters $\MBu, \MB{A}$.

For vectors $\Vu,\Vv\in\sR^d$ and matrices $\V{A},\V{B}\in\sR^{d\times d}$, we define the inner product for vectors $\Vu \cdot \Vv\in\sR$ and matrices $\V{A}:\V{B}\in\sR$, and the outer product of two vectors $\Vu \otimes \Vv\in\sR^{d\times d}$ as
\begin{align*}
\Vu \cdot \Vv &= \sum_{i=1}^d u_iv_i,
&
\V{A}:\V{B} &= \sum_{i,j=1}^dA_{ij}B_{ij},
&
\Vu \otimes \Vv &= (u_iv_j)_{i,j=1}^d.
\end{align*}

\section{DoGIP within the finite element method}\label{sec:dogip-within-finite-element-method}

\subsection{FEM on simplexes with Lagrange basis functions}\label{sec:fem-on-simplexes-with-lagrange-basis-functions}

The model problem is considered on an FEM mesh $\mesh(\Omega)$ composed of simplexes (e.g.\ triangles or tetrahedra) with standard properties (non-overlapping elements, no hanging nodes, regular shapes). Vector $\VN$ denotes a discretisation parameter.
The FEM spaces involve polynomials of order $k$ (denoting the highest possible order)
$$\sP_k = \{p:\sR^d\rightarrow\sR \sep p(\x) = \sum_{\V{\alp}\in\sN^d_0,\|\V{\alp}\|_1\leq k} a_{\V{\alpha}} x_1^{\alpha_1}x_2^{\alpha_2}\dotsc x_d^{\alpha_d}, a_{\V{\alpha}}\in\sR\}$$
on each element. Depending on the trial and test spaces (e.g. $H^1(\Omega)$, $L^2(\Omega)$, etc.), the FEM spaces differ in the continuity requirement on the facets\footnote{Facet is a mesh entity of codimension 1, i.e.\ edge in 2D and face in 3D.} of the elements. Here, we take into account continuous and discontinuous finite element spaces
\begin{align*}
\fC_{\VN,k} &= \{ v:\Omega\rightarrow\sR \sep v|_T \in \sP_k\quad\forall T\in\mesh(\Omega) \text{ and }v \text{ is continuous}\},
\\
\fD_{\VN,k} &= \{ v:\Omega\rightarrow\sR \sep v|_T \in \sP_k\quad\forall T\in\mesh(\Omega)\}.
\end{align*}

Each approximation function $v\in\fC_{\VN,k}$ can be expressed as a linear combination 
\begin{align*}
v(\x) = \sum_{i=1}^{\dim \fC_{\VN,k}} \Mv_i \phi^i(\x)
\end{align*}
of the Lagrange basis functions (shape functions) $\phi^i:\Omega\rightarrow\sR$.
The coefficients $\M{v_i}$, as the degrees of freedom, are the function values at the nodal points $\x^i_{\fC_{\VN,k}}$, i.e.\ $\M{v_i}=v(\x^i_{\fC_{\VN,k}})$. Those points are suitable for the interpolation of functions, thanks to the Dirac delta property of basis functions
\begin{align*}
\phi^j(\x^i_{\fC_{\VN,k}}) = \delta_{ij}.
\end{align*}

A technical problem arises in the case of discontinuous approximation space $\fD_{\VN,k}$ because the function values have jumps at the boundary of the elements.
Therefore, the function values of $v\in\fD_{\VN,k}$ at $\x^i_{\fD_{\VN,k}}$ are understood as a continuous extension from the element supporting the corresponding $i$-th basis function, i.e.
\begin{align*}
v(\x^i_{\fD_{\VN,k}}) = \lim_{\substack{\x\rightarrow \x^i_{\fD_{\VN,k}}\\\text{for } \x\in T_i}} v(\x),\quad
\text{where $T_i$ is an element supporting $\phi^i$}.
\end{align*}
In the latter text, the function values at the nodal points will be used in the above sense.

\emph{Reference element.} It is assumed that all of the elements $T$ are derived from a reference element $\hT$ using affine mapping $F_T(\hx):\hT\rightarrow T$ defined as $F_T(\hx)=R_T\hx+S_T$ with invertible matrices $R_T\in\sR^{d\times d}$ and vectors $S_T\in\sR^d$. Consistently the objects on a reference element are denoted with a hat.
Global basis functions $\phi_i:\Omega\rightarrow\sR$ of a FEM space $\sV$ are derived from a reference shape functions $\hphi_{\hat{i}}:\hT\rightarrow\sR$ by a composition with a mapping $F_T$, i.e. $\phi_i(\x) = \hphi_{\hi}(F_T^{-1}(\x))$ for $\x\in T$. 
The relation between local $\hat{i}$ and global index $i=l_T(\hat{i})$ is given with an injective map
$$l_T:\{1,\dotsc,V_T\}\rightarrow \{1,\dotsc,V\},\quad\text{where }V=\dim \sV\text{ and }V_T=\dim \sV|_T = \dim \sP_k.$$
The derivative of basis functions is related with the formula obtained by a chain rule
\begin{align}
\label{eq:bas_deriv}
\frac{\partial\phi(\x)}{\partial x_i} &= \frac{\partial \hphi(F^{-1}(\x))}{\partial x_i} = \sum_{p=1}^d \frac{\partial \hphi(F^{-1}(\x))}{\partial \hat{x}_p} (R_T^{-1})_{pi}.
\end{align}

Later on we also use a discretisation operator $\Di_\sV:\sV\rightarrow\sR^V$ that assigns to each approximation function from space $\sV$ a corresponding vector of coefficients, i.e.\
$\Di_\sV[v] = \MBv = (v(\x_\sV^i))_{i=1}^V$ with vectors denoted with upright sans-serif font.
Note that it is a one-to-one mapping between the space of functions $\sV$ and the space of coefficients $\sR^V$. Similarly we will also use the local degrees of freedom on an element $T$ as $\MBv_T = \bigl( \Mv_{l_T(\hi)}  \bigr)_{\hi=1}^{V_T}\in\sR^{V_T}$.

\subsection{Weighted projection}\label{sec:weighted-projection}

Here, the DoGIP is introduced for the problem of the weighted projection.
\begin{definition}[Weighted projection]
 Let $\sV$ be a finite dimensional approximation space, $m:\Omega\rightarrow\sR$ a uniformly positive integrable function representing e.g.\ a material field, and $f\in L^2(\Omega)$ a function representing e.g. a source term.
 Then we define the bilinear form $a:\sV\times\sV\rightarrow\sR$ and the linear functional $F:\sV\rightarrow\sR$ as
 \begin{align}
 \label{eq:bilf_gp}
 a(u,v) &= \int_\Omega m(\x) u(\x) v(\x) \D{\x},
 &
 F(v) &= \int_\Omega f(\x) v(\x) \D{\x}.
 \end{align}
 Then the weighted projection represents the problem:
 Find trial function $u\in\sV$ such that
 \begin{align*}
 a(u,v) = F(v)\quad\forall v\in\sV.
 \end{align*}
\end{definition}
In the standard approach, the solution of weighted projection $u=\sum_{i=1}^V \Mu_i \phi^i_\sV$  with $V=\dim \sV$ is expressed as a linear combination with respect to the basis of the space $\sV$. Then, the vector of coefficients $\MBu=(\Mu_i)_{i=1}^V$ is determined from the following square linear system
\begin{subequations}
 \label{eq:linsys_gp}
 \begin{align}
 \MB{A}\MBu &= \MB{b}\in\sR^V
 &\text{where }\M{A}_{ij} &= \int_\Omega m(\x) \phi^j_\sV(\x)\phi^i_\sV(\x)\D{\x},
 \\
 &&
 \M{b}_i &= \int_\Omega f(\x) \phi^i_\sV(\x)\D{\x}.
 \end{align}
\end{subequations}

The alternative DoGIP approach relies on the incorporation of the double grid space.
\begin{lemma}[Double-grid space]
 Let the functions $u,v$ be from a space $\sV=\fC_{\VN,k}$ and let $\sW=\fC_{\VN,2k}$.
 Then
 \begin{align*}
 u &\in\sW,
 &
 u v &\in\sW.
 \end{align*}
\end{lemma}
\begin{proof}
 The product $uv$ of two FEM functions is a continuous function and the polynomial order on each element is doubled.
 \end{proof}
\begin{theorem}[DoGIP for weighted projection]
 \label{lem:wp_dogip}
 Let $\sV=\fC_{\VN,k}$ be an approximation space and $\sW=\fC_{\VN,2k}$ be a double-grid space.
 Then the standard linear system matrix \eqref{eq:linsys_gp} can be decomposed into
 \begin{align*}
 \MB{A}\MBu\cdot\MBv &= \MB{B}^*\MB{A}{\diag} \MB{B}\MBu\cdot\MBv
 \quad\text{with }\MB{A}\diag\in\sR^{W\times W} \text{ and } \MB{B}\in\sR^{W\times V}
 \end{align*}
 where $V=\dim\sV$, $W=\dim\sW$, and the components of matrices are
 \begin{align*}
 \M{A}\diag_{ij} &= \delta_{ij}\int_\Omega m(\x)\phi^i_\sW(\x),
 &
 \M{B}_{ij} &= \phi^j_\sV(\x_\sW^i).
 \end{align*}
\end{theorem}
\begin{remark}
 In the case of the FEM space $\sV=\fC_{\VN,k}$ composed of the continuous functions, the double-grid space is of the same type but with the doubled polynomial order $\sW=\fC_{\VN,2k}$.
\end{remark}
\begin{proof}
 Since both trial $u$ and test function $v$ belong to the approximation space $\sV$, its product $uv$ belongs to $\sW$ and can be represented with respect to the Lagrangian basis of $\sW$ as
 \begin{align*}
 w(\x) = u(\x)v(\x) = \sum_{i=1}^W u(\x_\sW^i) v(\x_\sW^i)\phi^i_\sW(\x).
 \end{align*}
 The substitution of this formula into the bilinear form \eqref{eq:bilf_gp} reveals an effective evaluation with respect to the basis of $\sW$
 \begin{align*}
 \int_\Omega m(\x) u(\x) v(\x) \D{\x} &= \sum_{j=1}^W u(\x_\sW^j) v(\x_\sW^j) \int_\Omega m(\x)\phi^j_\sW(\x)\D{\x}
 \\
 &= \sum_{i,j=1}^W \underbrace{\delta_{ij}\int_\Omega m(\x)\phi^j_\sW(\x)\D{\x}}_{\MB{A}\diag_{ij}} \Mu^j_\sW \Mv_\sW^i = \scal{\M{A}\diag\MBu_\sW}{\MBv_\sW}_{\sR^M}
 \end{align*}
 where $\MB{A}\diag\in\sR^{W\times W}$ is a diagonal matrix and vectors $\MBu_\sW=\bigl(u(\x_\sW^j)\bigr)_{j=1}^W$ and $\MBv_\sW=\bigl(v(\x_\sW^j)\bigr)_{j=1}^W$ from $\sR^W$ store the coefficients of functions $u$ and $v$ at DOFs of $\sW$.
 
 Here one cannot derive directly a linear system because the test vectors $\MBv_\sW$ do not span the whole space $\sR^W$. Therefore we express the coefficients $\MBu_\sW$ and $\MBv_\sW$ in terms of DOFs of the original space, i.e.\ $\Mu_k=u(\x_\sV^k)$ and $\Mv_k=v(\x_\sV^k)$ for $k\in\{1,\dotsc,V\}$.
 
 The formula is based on the evaluation of the ansatz $u(\x) = 
 \sum_{j=1}^{V} u(\x_\sV^j) \phi^j_\sV(\x)$ at the DOFs of the double-grid space $\sW$, which reveals the interpolation-projection matrix $\MB{B}\in\sR^{W\times V}$ as
 \begin{align*}
 \Mu_\sW^i = u(\x_\sW^i) = 
 \sum_{j=1}^{V} u(\x_\sV^j) \phi^j_\sV(\x_\sW^i)
 =
 \sum_{j=1}^{V} \M{B}_{ij} \Mu_j
 \quad
 \text{for }i=1,\dotsc,W.
 \end{align*}
 The final formula in Theorem~\ref{lem:wp_dogip} is obtained by a substitution of the interpolation-projection matrix $\MB{B}$.
 \end{proof}
A special version of the DoGIP is obtained when the decomposition is provided on a level of elements.
\begin{theorem}[Element-wise DoGIP for weighted projection]
 \label{lem:wp_dogip_el}
 Let $\sV=\fC_{\VN,k}$ be an approximation space and $\sW=\fC_{\VN,2k}$ be a double-grid space.
 Then the quadratic form corresponding to the system matrix \eqref{eq:linsys_se} can be decomposed as
 \begin{align}
 \label{eq:wp_dogip_mul}
 \MB{A}\MBu\cdot\MBv &= \sum_{T\in \mesh} \hMBB^*\MB{A}_T\diag \hMBB \MBu_T\cdot\MBv_T = \sum_{T\in\mesh}\sum_{i,j=1}^{W_T} \sum_{k,l=1}^{V_T} \MA_{T,ij} \hMB_{jl}\Mu_{T,l}\hMB_{ik}\Mv_{T,k},
 \end{align}
 where $\MBu_T = \bigl(\Mu_{l_T(l)}\bigr)_{l=1}^{V_T}$, $\MBv_T = \bigl(\Mu_{l_T(k)}\bigr)_{k=1}^{V_T}$ stores the local DOFs on the element $T$.
 The arrays $\MBA_T\in\sR^{W_T\times W_T} \text{ and } \hMBB\in\sR^{W_T\times V_T}$ have the following components
 \begin{align*}
 \M{A}_{T,ij}\diag &= \del_{ij} \int_{T} m(\x)\phi^{l_T(j)}_\sW(\x)\D{\x},
 &
 \hMB_{jl} &= \hphi^l_\sV(\hx_\sW^j).
 \end{align*}
\end{theorem}
\begin{proof}The proof is analogical to the one of the previous theorem. The main idea is to split the integral over domain $\Omega$ to integrals over elements, i.e.\ 
 \begin{align}
 \label{eq:aux_wp_elem}
 \int_\Omega m(\x) u(\x) v(\x) \D{\x} = \sum_{T\in\mesh}\sum_{i,j=1}^W \del_{ij}\int_T m(\x)\phi^j_\sW(\x)\D{\x} \cdot u(\x_\sW^j) v(\x_\sW^i),
 \end{align}
 where the integrals are nonzero only when the basis functions $\phi^j_\sW$ have support on $T$.
 Moreover, the interpolation
 \begin{align*}
 u(\x_\sW^j) = 
 \sum_{l=1}^{V} u(\x^l_\sV) \phi^l_\sV(\x_\sW^j) =  \sum_{l=1}^{V} \Mu_l \hphi^l_\sV(F^{-1}_T(\x_\sW^j))
 =
 \sum_{l=1}^{V} \Mu_l \hphi^l_\sV(\hx_\sW^j),
 \end{align*}
 which is expressed with the reference basis, has nonzero coefficients only if $\x_\sW^i$ is contained in the support of $\phi^j_\sV$. The substitution of the interpolation into \eqref{eq:aux_wp_elem} and reparametrisation of indices $i,j,k,l$ to local indices leads to the formulas stated in the theorem.
 \end{proof}
\subsection{Scalar elliptic equation}
\label{sec:scalar-elliptic-equation}
Here, the DoGIP is described for a scalar elliptic equation modelling, e.g.\ diffusion or stationary heat transfer.

\begin{definition}[Scalar elliptic problem]
 Let $\sV$ be a finite dimensional approximation space, $f\in L^2(\Omega)$ be a function representing e.g. a source term, and $\T{M}:\Omega\rightarrow\sR^{d\times d}$ be an integrable, uniformly positive definite, and bounded matrix function, i.e.\
 \begin{align*}
 \exists c,C\in\sR \quad\text{such that}\quad
 0< c\leq\frac{\T{M}(\x) \y\cdot\y}{\|\y\|} \leq C
 \quad \forall\x\in\Omega\text{ and }\forall \y\in\sR^d.
 \end{align*}
 
 Then we define bilinear form $a:\sV\times\sV\rightarrow\sR$ and linear functional $F:\sV\rightarrow\sR$ as
 \begin{align}
 \label{eq:bilf_se}
 a(u,v) &= \int_\Omega \T{M}(\x) \nabla u(\x) \cdot \nabla v(\x) \D{\x},
 &
 F(v) &= \int_\Omega f(\x) v(\x) \D{\x}.
 \end{align}
 The Galerkin approximation of the scalar elliptic problem stands for:
 find trial function $u\in\sV$ such that
 \begin{align*}
 a(u,v) = F(v)\quad\forall v\in\sV.
 \end{align*}
\end{definition}

Analogically to the case in the previous section, the standard approach is based on the expression of the approximate solution $u=\sum_{i=1}^V \Mu_i \phi^i_\sV$ with $V=\dim\sV$ as a linear combination with respect to the basis of the space $\sV$.
The vector of coefficients $\MBu=(\Mu_i)_{i=1}^V$ is determined from the following square linear system
\begin{subequations}
 \label{eq:linsys_se}
 \begin{align*}
 \MB{A}\MBu &= \MB{b}\in\sR^V,
 &
 \text{where }
 \M{A}_{ij} &= \int_\Omega \T{M}(\x) \nabla\phi^j_\sV(\x)\cdot{\nabla}\phi^i_\sV(\x)\D{\x},
 \\
 &&\M{b}_i &= \int_\Omega f(\x) \phi^i_\sV(\x)\D{\x}.
 \end{align*}
\end{subequations}
The DoGIP approach incorporates the double-grid space described in the following lemma.
\begin{lemma}[Double-grid space]
 Let the functions $u,v$ be from the space $\sV=\fC_{\VN,k}$ and let $\sW=\fD_{\VN,2k-2}$ be the corresponding double-grid space.
 Then
 \begin{align*}
 \frac{\partial u}{\partial x_p}&\in\fD_{\VN,k-1}\subset\sW,
 &
 \frac{\partial u}{\partial x_p}\frac{\partial v}{\partial x_q}&\in\sW
 \quad\text{for }p,q\in\{1,\dotsc,d\}.
 \end{align*}
\end{lemma}
\begin{proof}
 The derivative of a function reduces the polynomial order by one while the continuity of the functions over facets (edges in 2D or faces in 3D) is lost.
 \end{proof}
Analogically to the weighted projection problem, there are two variants of the DoGIP approach: the global and the element-wise one. Since the double-grid space is discontinuous, the global variant has benefit only in the special cases of regular grid and isotropic materials. Both variants are presented in the following two theorems.
\begin{theorem}[DoGIP for elliptic problem]
 \label{lem:ep_dogip}
 Let $\sV=\fC_{\VN,k}$ be an approximation space and $\sW=\fD_{\VN,2k-2}$ be a double-grid space.
 Then the standard linear system matrix \eqref{eq:linsys_se} can be decomposed into
 \begin{align*}
 \MB{A}\MBu\cdot\MBv = \MB{B}^*\MB{A}{\diag} \MB{B}\MBu\cdot\MBv = \sum_{k,l=1}^V\sum_{r,s=1}^d \sum_{i,j=1}^W \M{A}\diag_{rsij}\M{B}_{sjl}  \Mu_l \M{B}_{rik}\Mv_k
 \end{align*}
 where $V=\dim\sV$, $W=\dim\sW$, and the matrices $\MB{A}\diag\in\sR^{d\times d\times W\times W} \text{ and } \MB{B}\in\sR^{d\times W\times V}$ have the following components
 \begin{align*}
 \M{A}\diag_{rsij} &= \delta_{ij}\int_\Omega M_{rs}(\x)\phi^j_\sW(\x)\D{\x}
 \quad \text{for }r,s\in\{1,\dotsc,d\}, i,j\in\{1,\dotsc,W\},
 \\
 \M{B}_{rik} &= \frac{\partial \phi^k_\sV(\x_\sW^i)}{\partial x_r}
 \quad \text{for }r\in\{1,\dotsc,d\}, i\in\{1,\dotsc,W\}, k\in\{1,\dotsc,V\}.
 \end{align*}
\end{theorem}
\begin{proof}
 Now we present a method based on the double-grid quadrature and interpolation operator. 
 Albeit both the trial $u$ and the test function $v$ belong to the approximation space $\sV$, we express them on the double-grid space $\sW^{d\times d}$ as
 \begin{align*}
 w(\x) = \nabla u(\x)\otimes \nabla v(\x) = \sum_{j=1}^W \nabla u(\x_\sW^j)\otimes\nabla v(\x_\sW^j)\phi^j_\sW(\x). 
 \end{align*}
 The substitution of this formula in \eqref{eq:bilf_se} already reveals an effective evaluation of the bilinear form
 \begin{align*}
 \int_\Omega \T{M}(\x) \nabla u(\x)\cdot \nabla v(\x) \D{\x} &=
 \int_\Omega \T{M}(\x) : [\nabla u(\x)\otimes \nabla v(\x) ] \D{\x}
 \\
 & = \sum_{j=1}^W \int_\Omega \T{M}(\x)\phi^j_\sW(\x)\D{\x} : \nabla u(\x_\sW^j) \otimes \nabla v(\x_\sW^j)
 \\
 &= \scal{\MB{A}\diag\MBu_\sW}{\MBv_\sW}_{\sR^{d\times W}},
 \end{align*}
 where $\MBu_\sW=(\nabla u(\x_\sW^j))_{j=1}^W$, $\MBv_\sW=(\nabla v(\x_\sW^j))_{j=1}^W\in\sR^{d\times W}$ stores the coefficients on the double-grid space, the integral defines the coefficient of the block-diagonal matrix $\MBA\diag$, and the matrix vector multiplication  is understood as
 \begin{align*}
 (\MB{A}\diag\MBu_\sW)_{ri} = \sum_{s=1}^d\sum_{j=1}^W\M{A}_{rsij}\diag\Mu_{\sW,sj}
 \quad\text{for }r=\{1,\dotsc,d\}\text{ and }i\in\{1,\dotsc,W\}.
 \end{align*}
 Here, we cannot derive directly a linear system because the test vectors $\MBv$ do not span the whole space $\sR^{d\times W}$.
 Therefore, we need an interpolation-projection operator between $\sV$ and $\sW^d$, or rather between their corresponding spaces of coefficients $\Di_\sV[\sV]$ and $\Di_\sW[\sW]^d$.
 The derivation of the matrix is based on the evaluation of the gradient of a function $u(\x)=\sum_{k=1}^V u(\x^k_\sV) \phi^k_\sV(\x)$ at the degrees of freedom of the double grid space, i.e.\ for $\x=\x_\sW^j$.
 It can be represented with an interpolation-projection operator $\MB{B}\in\sR^{d\times W\times V}$ as
 \begin{align*}
 \frac{\partial u(\x_\sW^j)}{\partial x_s}  &= 
 \sum_{l=1}^V u(\x_\sV^l) \underbrace{\frac{\partial \phi_\sV^l(\x_\sW^j)}{\partial x_s}}_{\M{B}_{sjl}} 
 \quad
 \text{for } s\in\{1,\dotsc,d\} \text{ and } j\in\{1,\dotsc,W\}.
 \end{align*}
 
 To summarise, the bilinear form can be expressed using different vectors of coefficients
 \begin{align*}
 a(u,v) &= \scal{\MB{A}\diag \MBu_\sW}{\MBv_\sW}_{\sR^{d\times W}} = \scal{\MB{A}\diag \MB{B} \MBu_\sV}{\MB{B}\MBv_\sV}_{\sR^{d\times W}} 
= \scal{\MB{B}^*\MB{A}\diag \MB{B} \MBu_\sV}{\MBv_\sV}_{\sR^{V}},
 \end{align*}
 where the $\MB{B}^*$ is the adjoint operator of $\MB{B}$.
 Since the last term holds for arbitrary vectors $\MBu_\sV$ and $\MBv_\sV$, the required decomposition is established.
 \end{proof}

The special case of the theorem arises for isotropic material coefficients:
\begin{corollary}
 Assume that $\T{M}(\x) = m(\x)\T{I}$ is proportional to the identity matrix $\T{I}$ with some uniformly positive and bounded function $m:\Omega\rightarrow\sR$.
 Then, the coefficients of the block-diagonal matrix from previous theorem simplify to
 \begin{align*}
 \M{A}\diag_{rsij} &= \delta_{rs}\delta_{ij}\int_\Omega m(\x)\phi^j_\sW(\x)\D{\x}.
 \end{align*}
 In particular for $d=2$, the matrix $\MB{A}\diag$ can be expressed as block-diagonal operator
 \begin{align*}
 \MB{A}\diag &= 
 \begin{pmatrix}
 \MB{a}\diag & \MB{0} \\
 \MB{0} & \MB{a}\diag
 \end{pmatrix},
 &
 \MB{a}\diag &= \biggl(\delta_{ij}\int_{\Omega} m(\x)\phi^i_\sW(\x)\D{\x}\biggr)_{i,j=1}^W.
 \end{align*}
\end{corollary}
\begin{theorem}[Element-wise DoGIP for the elliptic problem]
 \label{lem:ep_dogip_el}
 Let $\sV=\fC_{\VN,k}$ be an approximation space and $\sW=\fD_{\VN,k}$ be a double-grid space.
 Then the quadratic form corresponding to the linear system \eqref{eq:linsys_se} can be decomposed as
 \begin{align*}
 \MB{A}\MBu\cdot\MBv &= \sum_{T\in \mesh} \hMBB^*\MB{A}_T\diag \hMBB \MBu_T\cdot\MBv_T,
 \end{align*}
 where $\MBu_{T,l} = \Mu_{l_T(l)}, \MBv_{T,k} = \Mv_{l_T(k)}$ stores the local DOFs on the element $T$. The quadratic forms on the element level are evaluated as
 \begin{align}
 \label{eq:ep_dogip_mul}
 \hMBB^*\MB{A}_T\diag \hMBB \MBu_T\cdot\MBv_T &= \sum_{r,s=1}^d \sum_{i,j=1}^{W_T} \sum_{k,l=1}^{V_T} \MA_{T,rsij} \hMB_{sjl}\Mu_{T,l}\hMB_{rik}\Mv_{T,k},
 \end{align}
 where the arrays $\MBA\diag_T\in\sR^{d\times d\times W_T\times W_T} \text{ and } \hMBB\in\sR^{d\times W_T\times V_T}$ have the following components
 \begin{align*}
 \M{A}\diag_{T,rsij} &= \del_{ij}\sum_{p,q=1}^d R^{-1}_{T,rp} R^{-1}_{T,sq} \int_{T} M_{pq}(\x)\phi^i_\sW(\x)\D{\x},
 &
 \hMB_{rik} &= \frac{\partial \hphi^k_\sV(\hx_\sW^i)}{\partial x_r}.
 \end{align*}
\end{theorem}
\begin{proof}
 Now we present the method based on the double-grid quadrature and interpolation operators. 
 Albeit both the trial $u$ and the test function $v$ belong to the approximation space $\sV$, we express their tensor product $u\otimes v$ on the double-grid space $\sW^{d\times d}$ as
 \begin{align*}
 w(\x) = \nabla u(\x)\otimes \nabla v(\x) = \sum_{j=1}^W \nabla u(\x_\sW^j)\otimes\nabla v(\x_\sW^j)\phi^j_\sW(\x). 
 \end{align*}
 The substitution of this formula into the bilinear form reveals an effective evaluation
 \begin{align*}
 \int_\Omega
 \T{M}(\x) \nabla u(\x)\cdot \nabla v(\x) \D{\x}
 &=
 \sum_{T\in\mesh} \int_T \T{M}(\x) : [\nabla u(\x)\otimes \nabla v(\x) ] \D{\x}
 \\
 & = \sum_{T\in\mesh} \sum_{j=1}^W  \int_T \T{M}(\x)\phi^j_\sW(\x)\D{\x} : \nabla u(\x_\sW^j) \otimes \nabla v(\x_\sW^j)
 \\
 & = \sum_{T\in\mesh} \sum_{i,j=1}^W \sum_{p,q=1}^d \del_{ij}\int_{T} M_{pq}(\hx)\phi^j_\sW(\x)\D{\x} 
 \cdot \frac{\partial u(\x_\sW^j)}{\partial x_p} \frac{\partial v(\x_\sW^i)}{\partial x_q},
 \end{align*}
 which is written here as a weighted scalar product between arrays $\bigl(\nabla u(\x_\sW^j)\bigr)_{j=1}^W$ and $\bigl(\nabla v(\x_\sW^i)\bigr)_{i=1}^W$.
 
 Now, these components will be expressed in the terms of nodal values $\Mu_k = v(\x_\sV^k)$ of the space $\sV$.
 The formula is based on the interpolation of the gradient of a function $u(\x)=\sum_{k=1}^V \Mu_k \phi^k_\sV(\x)$ at the DOFs of the double grid space $\sW$ (i.e.\ at $\x=\x_\sW^j$ with basis $\phi_\sW^j$ supported at the element $T$)
 \begin{align*}
 \frac{\partial u(\x_\sW^j)}{\partial x_q}  &= 
 \sum_{l=1}^V \Mu_l \frac{\partial \phi_\sV^l(\x_\sW^j)}{\partial x_q}
 = 
 \sum_{l=1}^V \Mu_l \sum_{s=1}^d \frac{\partial \hphi_\sV^l(\hx_\sW^{j})}{\partial \hat{x}_s} (R_{T}^{-1})_{sq},
 \end{align*}
 where the derivatives of the basis functions are expressed with respect to the reference basis function, see \eqref{eq:bas_deriv}, and $\hx_\sW^{j} = F_T^{-1}(\x_\sW^{j})$.
 
 The substitution into the integral provides the DoGIP evaluation over the finite elements
 \begin{align*}
 &\int_\Omega
 \T{M}(\x) \nabla u(\x)\cdot \nabla v(\x) \D{\x}=
 \\
 &= \sum_{T\in\mesh} \sum_{i,j=1}^W \sum_{r,s=1}^d \left( \del_{ij}\sum_{p,q=1}^d R^{-1}_{T,rp} R^{-1}_{T,sq} \int_{T} M_{pq}(\x)\phi^j_\sW(\x)\D{\x} \right)
 \cdot \sum_{k,l=1}^V  \Mu_l \Mv_k  \frac{\partial \hphi_\sV^l(\hx_\sW^j)}{\partial x_s}  \frac{\partial \hphi_\sV^k(\hx_\sW^i)}{\partial x_r}.
 \end{align*}
 The interpolation coefficients $\hphi_\sV^k(\hx_\sW^i)$ as well as the integral over $T$ are nonzero only if the support of $\phi_\sV^k$ and $\phi^i_\sW$ is on the same element $T$.
 Therefore, we can reparametrise the sums over $i,j,k$, and $l$ to local indices in order to obtain the final formula stated in the theorem.
 \end{proof}

The actual matrix-vector multiplication can be provided by the following Algorithm~ \ref{alg:mat-vec_mult}.
\begin{algorithm}
 \caption{
  Matrix-vector multiplication for element-wise DoGIP}
 \label{alg:mat-vec_mult}
 \begin{algorithmic}[1]
  \Require
  {$\MBA_T\diag\in\sR^{d\times d\times W_T\times W_T}$ for all $T\in\mesh$, $\hMBB\in\sR^{d\times W_T\times V_T}$
  }
  \Function{multiplication}{$\MBA$, $\MBu$}
  \State{allocate $\MBv$ of size like $\MBu$}
  \For{$T$ in $\mesh$}
  \Comment{iteration over all elements}
  \State{get local DOFs $\MBu_T$ from $\MBu$}
  \State{calculate sequentially from right $\hMBB\MBA_T\diag\hMBB\MBu_T$}
  \Comment{Eq.~\eqref{eq:ep_dogip_mul}}
  \State{add $\hMBB\MBA_T\diag\hMBB\MBu_T$ to $\MBv$}
  \EndFor{}
  \State{\Return $\MBv$}
  \EndFunction
 \end{algorithmic}
\end{algorithm}

\section{Numerical results}\label{sec:numerical-results}
Here the memory and computational requirements are compared for the DoGIP and the standard discretisation approach.
The two variational problems presented in sections~\ref{sec:weighted-projection} and \ref{sec:scalar-elliptic-equation} were implemented in an open-source FEM software \href{https://fenicsproject.org/}{FEniCS} \cite{Fenics_v15};
Python scripts are freely available on \href{https://github.com/vondrejc/DoGIP}{https://github.com/vondrejc/DoGIP}.

The irregular meshes $\meshN$, with an example depicted in Figure~\ref{fig:mesh}, are described with the parameter $\VN=(N,N,N)$ corresponding to the number of elements in each spatial direction;
for simplicity the mesh will be denoted as $\mesh$.

\begin{figure}[htp]
 \centering
 \includegraphics[width=0.4\linewidth]{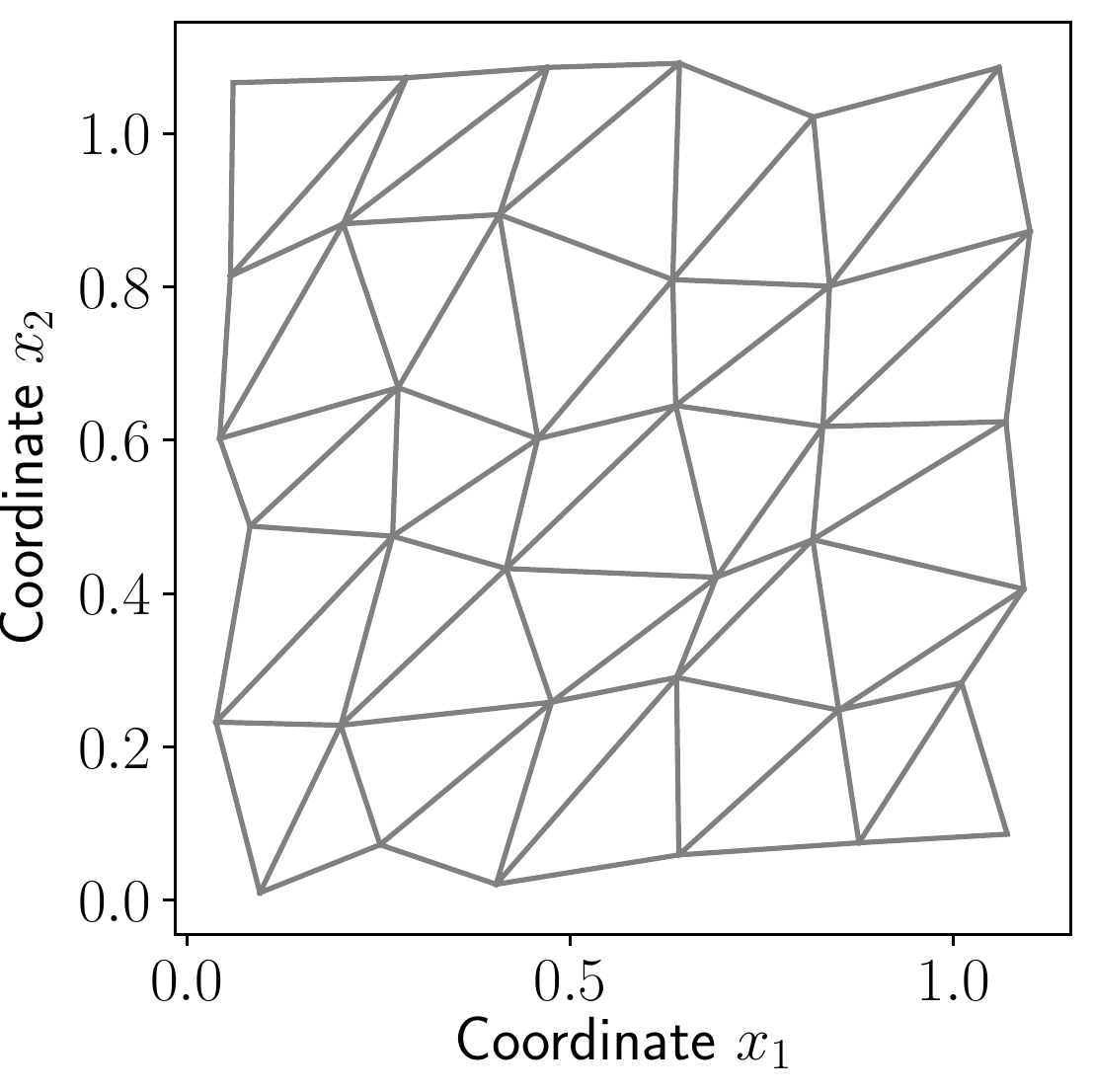}
 \caption{Finite element mesh $\mathcal{M}_\VN$ of size $\VN=(5,5)$ and characteristic mesh size $\frac{1}{5}$.}
 \label{fig:mesh}
\end{figure}

Now the memory and computational demands are discussed.
To store sparse matrices obtained from FEM formulations, the compressed sparse row (CSR) format is considered with memory requirements
\begin{align*}
\mem \MB{A} = 2\nnz \MB{A} + \nrows \MB{A},
\end{align*}
where $\nnz$ is a number of non-zero elements and $\nrows$ number of rows.
This format requires to store all the nonzero elements, all the corresponding column indices in each row, and the number of elements in each row.
It is also noted that the values below the threshold $10^{-14}$ are consistently considered as zeros and therefore do not affect the memory requirements.

The following ratio provides
\begin{align}
\label{eq:memory_eff}
\text{memory efficiency} = \frac{\nnz \MB{A}_T\diag\cdot\dim\mesh}{\mem \MB{A}}
\end{align}
between DoGIP and standard approach, where $\nnz\MB{A}_T\diag$ are non-zero elements of dense matrices $\MB{A}_T\diag$ for each element $T$, and $\dim\mesh$ corresponds to the number of elements ($2N^2$ in 2D and $6N^3$ in 3D).

The computational requirements are directed by the number of operations (real multiplications) for matrix-vector multiplication which is directed by the number of nonzero elements of the system matrix $\MB{A}$ or item-by-item multiplication with matrix $\hMBB^*\MB{A}_T\diag\hMBB$ for each element $T$ of the mesh $\mesh$. The following ratio provides
\begin{align}
\label{eq:comp_eff}
\text{computational eff.} = \frac{[2(\nnz\hMBB-\nnz_{\pm1}\hMBB)+\nnz\MB{A}_T\diag]\cdot\dim\mesh}{\nnz\MB{A}}
\end{align}
between DoGIP and standard FEM.
The multiplication in DoGIP can benefit from the regular structure with multiplication with $\hMBB$ that could be optimised for a fast evaluation.
Additional costs that are attributed to the multiplication with sparse matrices are not considered here.

\subsection{Weighted projection}
In the weighted projection, the basis space $\sV$ is considered to be the space of continuous piecewise-polynomials $\fC_{\VN,k}$.
Then the double-grid space can be considered again as a space of continuous piecewise-polynomials but with the doubled polynomial order, i.e.\
\begin{align*}
\sV &= \fC_{\VN,k},
&
\sW &= \fC_{\VN,2k}.
\end{align*}

\begin{table}[ht]
 \caption{Comparison of the DoGIP and standard approach for weighted projection in 2D with $\dim\sV=1,442,401$ for all cases in this table. Here $1/N$ is the characteristic mesh size, $k$ is a polynomial order, and memory and computational efficiency are defined in \eqref{eq:memory_eff} and \eqref{eq:comp_eff}.}
 \label{tab:WP_2D}
 \centering
\begin{tabular}{|rr|rr|rrr|rr|} 
 \hline
 \multicolumn{2}{|c|}{param.} & \multicolumn{2}{|c|}{standard FEM} & \multicolumn{3}{|c|}{DoGIP} &  \multicolumn{2}{|c|}{effectiveness} \\
 $N$ & $k$ & $\mem \MB{A}$ & $\mem \MB{A}_T$ & $\mem \MB{A}\diag$ & $\mem \MB{A}\diag_T$ & $\nnz \hMBB$ &  mem. & comp.  \\
 \hline
  1,200 & 1 & 21,616,803 & 9 & 17,280,000 & 6 & 9 & 0.80 & 5.14  \\ 
  600 & 2 & 34,581,447 & 36 & 10,800,000 & 15 & 39 & 0.31 & 3.52  \\ 
  400 & 3 & 50,426,139 & 100 & 8,960,000 & 28 & 115 & 0.18 & 3.11  \\ 
  300 & 4 & 69,150,763 & 225 & 8,100,000 & 45 & 270 & 0.12 & 2.95  \\ 
  240 & 5 & 90,755,463 & 441 & 7,603,200 & 66 & 546 & 0.08 & 2.86  \\ 
  200 & 6 & 115,239,527 & 784 & 7,280,000 & 91 & 994 & 0.06 & 2.84  \\ 
  150 & 8 & 172,850,259 & 2,025 & 6,885,000 & 153 & 2,655 & 0.04 & 2.82  \\ 
 \hline
\end{tabular}
\end{table}

\begin{table}[ht]
 \caption{Comparison of the DoGIP and standard approach for weighted projection in 3D with $\dim\sV=912,673$ for all cases in this table. Here $1/N$ is the characteristic mesh size, $k$ is a polynomial order, and memory and computational efficiency are defined in \eqref{eq:memory_eff} and \eqref{eq:comp_eff}.}
 \label{tab:WP_3D}
 \centering
 \begin{tabular}{|rr|rr|rrr|rr|} 
  \hline
  \multicolumn{2}{|c|}{param.} & \multicolumn{2}{|c|}{standard FEM} & \multicolumn{3}{|c|}{DoGIP} &  \multicolumn{2}{|c|}{effectiveness} \\
  $N$ & $k$ & $\mem \MB{A}$ & $\mem \MB{A}_T$ & $\mem \MB{A}\diag$ & $\mem \MB{A}\diag_T$ & $\nnz \hMBB$ &  mem. & comp.  \\
  \hline
  96 & 1 & 27,843,551 & 16 & 53,084,160 & 10 & 16 & 1.91 & 13.40  \\ 
  48 & 2 & 52,423,011 & 100 & 23,224,320 & 35 & 116 & 0.44 & 6.36  \\ 
  32 & 3 & 87,684,871 & 400 & 16,515,072 & 84 & 520 & 0.19 & 4.91  \\ 
  24 & 4 & 135,540,615 & 1,225 & 13,685,760 & 165 & 1,729 & 0.10 & 4.38  \\ 
  \hline
 \end{tabular}
\end{table}

In the numerical test summarised in Tables~\ref{tab:WP_2D} and \ref{tab:WP_3D}, the number of elements N (or the characteristic mesh size $1/N$) and polynomial order $k$ have been considered as parameters.
They are chosen such that they provide the same dimension of the original space $\sV$.
One directly observes that the sparsity of the original systems $\MB{A}$ is significantly reduced with the growing polynomial order $k$.
On the other side, the memory requirements to store the DoGIP matrix $\MBA\diag$ are reduced compared to the conventional approach, which makes the DoGIP approach more efficient, especially for higher order polynomials.
Particularly, the memory demands for DoGIP in 2D are always smaller (regardless the polynomial order) than for the standard approach with linear polynomials; in 3D it is valid for DoGIP with polynomial order two and higher.
The memory effectiveness reaches the value of $0.04$ in 2D for $k=8$ and the value of $0.10$ in 3D for $k=4$.
On the other side the computational demands are increased for the DoGIP approach but decrease with increasing polynomial order $k$.
The values of computational efficiency, see \eqref{eq:comp_eff}, reach $2.81$ for $k=8$ in 2D and $4.25$ for $k=4$ in 3D.

Note that the memory requirements for the DoGIP approach decrease with increasing polynomial order.
It can be explained by the fact that the double grid space is continuous with overlapping degrees of freedom at the interface of the elements, but the DoGIP approach stores the matrices $\MBA_T\diag$ element-wise.
Therefore better results can be achieved for a global variant stated in Theorem~\ref{lem:ep_dogip} because the continuity of the double grid space can be incorporated.
This could reduce the memory and computationally requirements for smaller polynomial orders.

\subsection{Scalar elliptic equation}
In the case of scalar elliptic equation, the initial approximation space is again considered as a space of continuous piece-wise polynomials; however, the double-grid space consists of discontinuous piece-wise polynomials with reduced polynomial order $2(k-1)$ because of the gradient, i.e.\
\begin{align*}
\sV &= \fC_{\VN,k},
&
\sW &= \fD_{\VN,2(k-1)}.
\end{align*}

\begin{table}[ht]
 \caption{Comparison of the DoGIP and standard approach for scalar elliptic equation with isotropic material coefficients in 2D with $\dim\sV=1,442,401$ for all cases in this table. Here $1/N$ is the characteristic mesh size, $k$ is a polynomial order, and memory and computational efficiency are defined in \eqref{eq:memory_eff} and \eqref{eq:comp_eff}.}
 \label{tab:SE_2D}
 \centering
  \begin{tabular}{|rr|rr|rrr|rr|} 
   \hline
   \multicolumn{2}{|c|}{param.} & \multicolumn{2}{|c|}{standard FEM} & \multicolumn{3}{|c|}{DoGIP} &  \multicolumn{2}{|c|}{effectiveness} \\
   $N$ & $k$ & $\mem \MB{A}$ & $\mem \MB{A}_T$ & $\mem \MB{A}\diag$ & $\mem \MB{A}\diag_T$ & $\nnz \hMBB$ &  mem. & comp.  \\
   \hline
  1,200 & 1 & 21,616,803 & 9 & 11,520,000 & 4 & 4 & 0.53 & 1.14  \\ 
600 & 2 & 34,581,603 & 36 & 17,280,000 & 24 & 44 & 0.50 & 3.13  \\ 
400 & 3 & 50,426,403 & 100 & 19,200,000 & 60 & 212 & 0.38 & 5.80  \\ 
300 & 4 & 69,151,203 & 225 & 20,160,000 & 112 & 612 & 0.29 & 6.81  \\ 
240 & 5 & 90,756,003 & 441 & 20,736,000 & 180 & 1,516 & 0.23 & 8.10  \\ 
200 & 6 & 115,240,803 & 784 & 21,120,000 & 264 & 2,992 & 0.18 & 8.66  \\ 
150 & 8 & 172,850,403 & 2,025 & 21,600,000 & 480 & 9,232 & 0.12 & 9.88  \\ 
   \hline
  \end{tabular}
\end{table}

\begin{table}[ht]
 \caption{Comparison of the DoGIP and standard approach for scalar elliptic equation with isotropic material coefficients in 3D with $\dim\sV=912,673$ for all cases in this table. Here $1/N$ is the characteristic mesh size, $k$ is a polynomial order, and memory and computational efficiency are defined in \eqref{eq:memory_eff} and \eqref{eq:comp_eff}.}
 \label{tab:SE_3D}
 \centering
 \begin{tabular}{|rr|rr|rrr|rr|} 
  \hline
  \multicolumn{2}{|c|}{param.} & \multicolumn{2}{|c|}{standard FEM} & \multicolumn{3}{|c|}{DoGIP} &  \multicolumn{2}{|c|}{effectiveness} \\
  $N$ & $k$ & $\mem \MB{A}$ & $\mem \MB{A}_T$ & $\mem \MB{A}\diag$ & $\mem \MB{A}\diag_T$ & $\nnz \hMBB$ &  mem. & comp.  \\
  \hline
   96 & 1 & 27,843,555 & 16 & 47,775,744 & 9 & 6 & 1.72 & 3.55  \\ 
  48 & 2 & 52,423,203 & 100 & 59,719,680 & 90 & 126 & 1.14 & 6.03  \\ 
  32 & 3 & 87,773,283 & 400 & 61,931,520 & 315 & 1,014 & 0.71 & 9.79  \\ 
  24 & 4 & 135,589,539 & 1,225 & 62,705,664 & 756 & 4,590 & 0.46 & 11.82  \\ 
  \hline
 \end{tabular}
\end{table}

The numerical results are summarised in Tables~\ref{tab:SE_2D} and \ref{tab:SE_3D} again for two parameters (the characteristic mesh size $1/N$ and polynomial order $k$), which are chosen in a way to provide the same dimension of the original space $\sV$.
Contrary to the weighted projection in the previous subsection, the size of the double-grid space $\sW$ and thus the size of block-diagonal matrix $\MB{A}\diag$ grow with polynomial order $k$; however, the rate is slower than the memory requirements on the original matrix $\MB{A}$, which can be seen from the decreasing values of memory effectiveness.
Moreover, the memory requirements for DoGIP in 2D are always smaller (regardless the polynomial order) than for the standard approach with linear polynomials.
In 3D it is not observed any more, however, the DoGIP is always more memory efficient than the standard approach with polynomial order three.
Particularly the memory effectiveness reaches $0.12$ for $k=8$ in 2D and $0.45$ for $k=4$ in 3D.
Still the memory requirements fail to be beneficial for polynomial order $1$ and $2$ in 3D.
The memory requirements grow with polynomial order $k$ reaching the factor of $10$.

\subsection{Comparison to other matrix-free approaches}
In this section, the DoGIP approach that decomposes the system matrix into the interpolation and element-wise multiplication is compared to the existing matrix-free approaches \cite{Orszag1980,Kronbichler2012,Cantwell2011,Deville2002book,Huismann2017}.
Their difference lies in the interpolation operator.
The conventional matrix-free approaches interpolate over integration points while the DoGIP approach interpolates the trial functions on the double-grid space, which is spanned by the product of the approximation and the test space (possibly with a differential operator).

Therefore the DoGIP approach is beneficial whenever the dimension of the double grid space $\sW$ is smaller than the number of integration points.
In Figure~\ref{fig:nip}, the dimensions of the approximation and the double-grid space are depicted for different polynomial order $k$ along with the number of integration points of the Gauss type.
In the case of weighted projection, the product of the trial and the test function is a polynomial of the doubled order, i.e.\ $2k$.
It allows to choose a numerical integration, which is exact for a constant material coefficient. Following \cite{Witherden2015}, the number of quadrature points is significantly lower than the dimension of the double grid space $\sW$.

For sufficiently smooth material coefficients, the Gauss-type integration can still integrate the integrands accurately, however it may require to use higher order integration schemes.
For the integration that is exact for polynomials up to order $3k$, an efficient numerical integration on the simplexes can be chosen especially for lower order polynomials; higher order integration rules are more difficult to compute.
As an example, the number of integration points for the scheme in \cite{Karniadakis2013}, which is implemented in the software FEniCS for higher order polynomials, is comparable to the dimension of the double grid space in 2D but higher in 3D.

A totally different situation arises for nonsmooth integrands, which calls for some alternatives to Gauss-like integration rules.
Jumps in the functions naturally occur, e.g.\ when the material coefficients are heterogeneous or when the crack path (or the boundary of the computational domain) intersects a finite element.
In those situations a special integration technique has to be used as in the finite cell method \cite{Duster2008,Yang2012,Abedian2013b} incorporating an adaptive approach to decompose the element into subelements, which are integrated separately.
As these techniques require a significantly higher number of integration points than Gauss-like methods, the DoGIP approach provides an advantage compared to the conventional matrix-free approaches.

\begin{figure}[htb]
 \centering
 \subfloat[Weighted projection in 2D]{\includegraphics[scale=0.48]{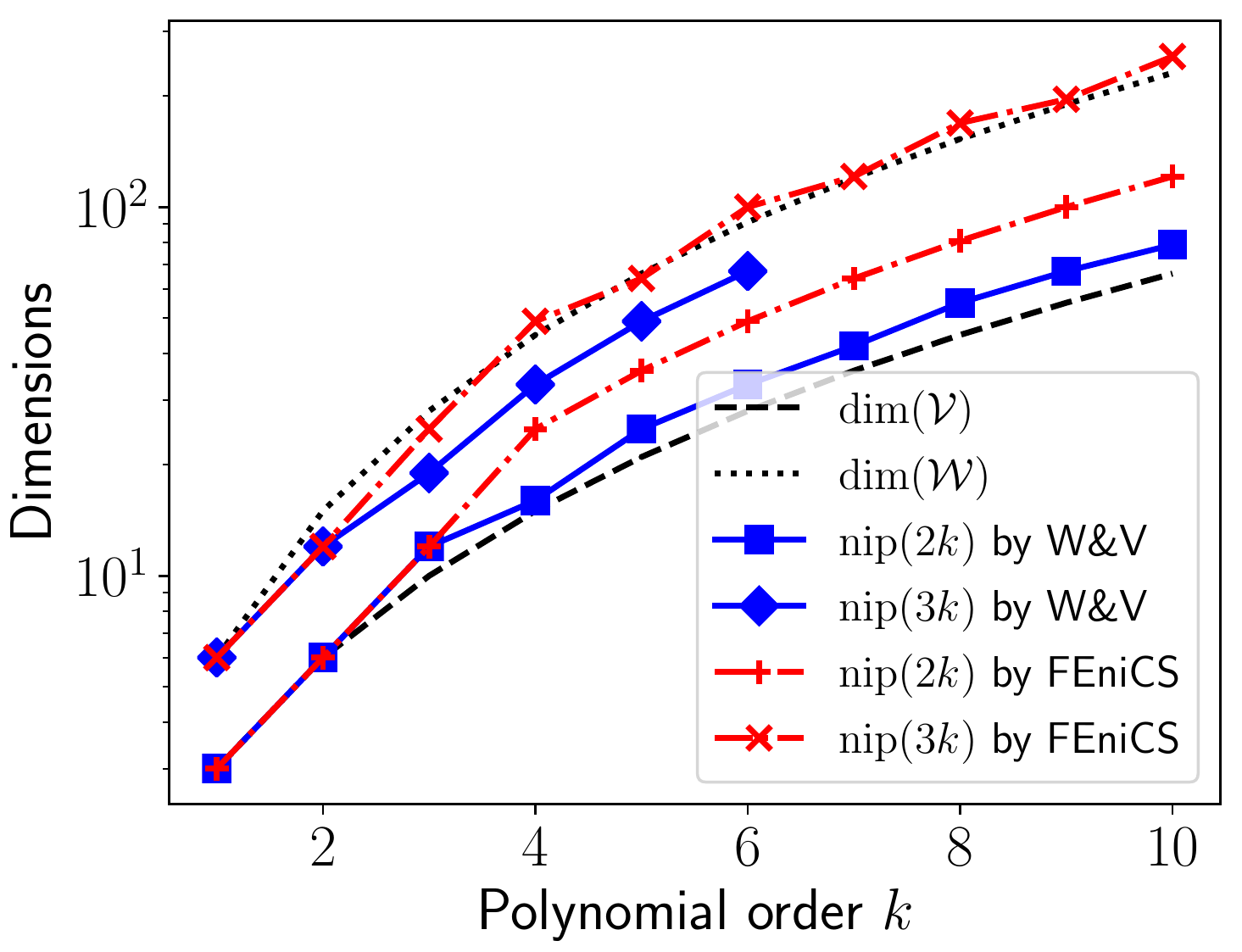}}
 \subfloat[Weighted projection in 3D]{\includegraphics[scale=0.48]{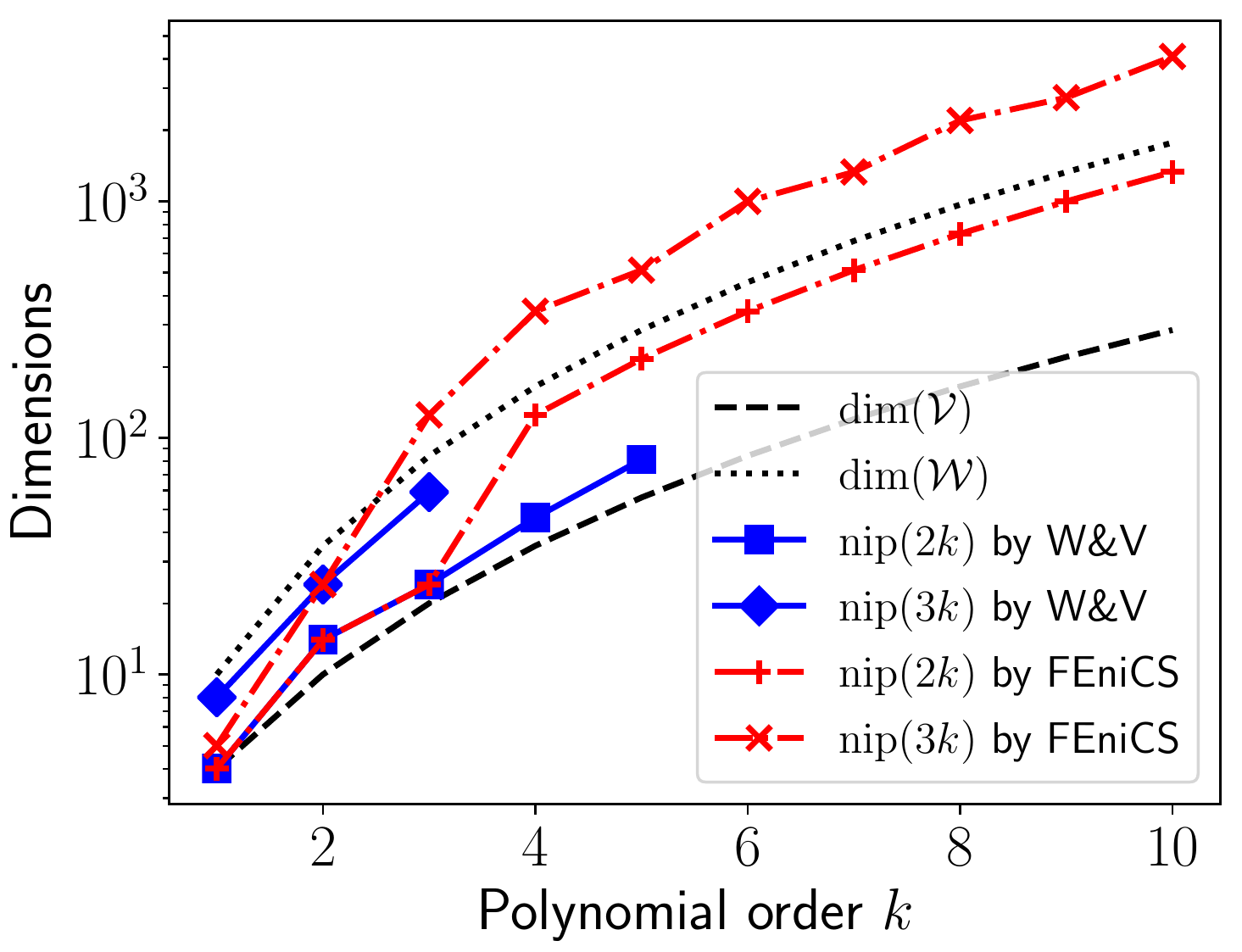}}
 \caption{Dimensions of approximation spaces and of integration points on a single element. The values $\mathrm{nip}(2k)$ and $\mathrm{nip}(3k)$ denote the number of integration points for a rule that can integrate exactly polynomials of order $2k$ or $3k$, resp.
  The rule W\&V (Witherden \&
  Vincent) denotes for a regular integration rule presented in \cite{Witherden2015}.
The integration rule implemented in FEniCS \cite{LoggWells2010a,Logg2007a,OlgaardWells2010b} differs for higher polynomial orders as it incorporates the collapsed integration rule presented in \cite{Karniadakis2013} by Karniadakis and Sherwin.
}
 \label{fig:nip}
\end{figure}

\section{Conclusion and discussion}

In this paper, the double-grid integration with interpolation-projection (DoGIP) is introduced and compared to the standard Galerkin discretisation of variational formulations.
This novel matrix-free discretisation approach introduced for the Fourier--Galerkin method in \cite{VoZeMa2014GarBounds,Vondrejc2015FFTimproved} is here recognised to be a more general framework.
Here, it is investigated within the finite element method (FEM) on simplexes for the weighted projection and the scalar elliptic equation.

The main results and observations are summarised:
\begin{itemize}
 \item The DoGIP method is based on a decomposition of the original system matrix into a (block) diagonal $\MB{A}\diag$ and an interpolation-projection matrix $\hMBB$ evaluated only on a reference element; interpolation is thus independent of the material coefficients or mesh distortions.
 
 \item The DoGIP approach is typically memory efficient compared to the original system, however the computational demands for FEM on simplexes are higher.
 
 \item The effectiveness of the traditional matrix-free approaches incorporating interpolation over quadrature points depends on the required number of integration points.
 The DoGIP splits the quadrature and interpolation into separate tasks.
 Therefore it is suitable for general problems when a higher number of integration points have to be used, e.g. for material coefficients that are discontinuous over elements.
 For problems with constant or smooth material coefficients, the standard matrix-free approaches remain the first option.
\end{itemize}

The proposed DoGIP approach has several possibilities for future research, which is discussed here:
\begin{itemize}
 \item Reducing the complexity of the interpolation-projection operators $\MB{B}$ by a special choice of the basis functions for a primal as well as a double-grid space.
 \item Application of the method to different discretisations and approximations such as FEM on quadrilaterals, spectral methods, or wavelets.
 \item Effective evaluation of interpolation-projection operator $\MB{B}$ using e.g.\ low-rank approximations.
 \item Investigation of special linear solvers such as the multigrid that could fit to the structure of the DoGIP systems.
\end{itemize}

\subsubsection*{Acknowledgements}
This paper was funded by the Deutsche Forschungsgemeinschaft (DFG, German Research Foundation) (project number MA2236/27-1) and the Czech Science Foundation (project number 17-04150J).


\end{document}